\documentclass[10pt,activeacute,amsfonts,reqno]{amsart}
\usepackage{a4wide}
\usepackage[english]{babel}
\usepackage{inputenc}
\usepackage{tikz}
\usepackage{mathabx}
\usepackage{subfigure}
\usepackage{amsmath,amsthm,amsfonts,amssymb,amscd}
\usepackage{amssymb,amsfonts}
\usepackage[all,arc]{xy}
\usepackage{enumerate}
\usepackage{mathrsfs}
\usepackage{marginnote}
\usepackage{mathtools}
\usepackage{amsthm,thmtools,xcolor}

\usepackage[all]{xy}
\usepackage{tikz-cd}
\usetikzlibrary{cd}
\usepackage[square,sort,comma,numbers,nonamebreak]{natbib}
\usetikzlibrary{decorations.pathreplacing}

\theoremstyle{plain}
\newtheorem{thm}{Theorem}[section]

\newtheorem{lem}{Lemma}[section]

\newtheorem{rem}{\textit{Remark}}[section]

\newtheorem{lemx}{Lemma}

\theoremstyle{definition}
\newtheorem{defn}{Definition}[section]

\theoremstyle{remark}

\numberwithin{equation}{section}
\usepackage{color}

\usepackage[colorlinks = true,
linkcolor = blue,
urlcolor  = blue,
citecolor = blue,
anchorcolor = blue]{hyperref}
\usepackage{hyperref}
\numberwithin{equation}{section}

\DeclareMathOperator{\sgn}{\mathrm{sgn}}
\DeclareMathOperator{\supp}{\mathrm{supp}}

\newcommand{\vertiii}[1]{{\left\vert\kern-0.25ex\left\vert\kern-0.25ex\left\vert #1 
    \right\vert\kern-0.25ex\right\vert\kern-0.25ex\right\vert}}
\newcommand\underrel[3][]{\mathrel{\mathop{#3}\limits_{%
			\ifx c#1\relax\mathclap{#2}\else#2\fi}}}

\providecommand{\norm}[1]{\left\lVert#1\right\rVert}

\title[Minimal Blow-up rate in 2D ZK]{On the minimal Blow-up rate for the 2D generalized Zakharov- Kuznetsov model}

\author[J. Trespalacios]{Jessica Trespalacios}
\address{Departamento de Ingeniería Matemática, Universidad de Chile, Santiago, Chile}
\email{jtrespalacios@dim.uchile.cl}
\thanks{J.T.'s work was funded in part by the National Agency for Research and Development (ANID)/ DOCTORADO NACIONAL/2019 -21190604, Chilean research grants FONDECYT 1231250, Centro de Modelamiento Matemático (CMM), ACE210010 and FB210005, BASAL funds for centers of excellence from ANID-Chile.}

\subjclass{Primary: 35Q76. Secondary: 35Q75}

\keywords{Blow up, modified Zakharov-Kuznetsov equation}	
\date{\today}

\begin{document}
	\begin{abstract} In this note we consider the generalized Zakharov-Kuznetsov equation in $\mathbb R^2$, for initial conditions in the Sobolev space $H^s$ with $s>3/4.$ Assuming that there is a blow-up solution at finite time $T^{*}$, we obtain a lower bound for the blow-up rate of that solution, expressed in terms of a lower bound for the $H^s$ norm of the solution. In the particular case of the modified Zakharov-Kuznetsov equation,  {\color{teal} a nontrivial gap is found between conjectured blow-up rates and our results.} The analysis is based on properly quantifying the linear estimates given by Faminskii \cite{Faminskii}, as well as the local well-posedness theory of Linares and Pastor \cite{Linares2009,LinaresPastor}, combined with an argument developed by  Weissler \cite{Weissler} and {\color{teal} Colliander, Czuback and Sulem} \cite{Colliander} in the context of the semilinear heat equations.
			\end{abstract}
	
	\maketitle
	
	
\section{ Introduction and main results}
\subsection{Setting}
Consider the two-dimensional (2D) generalized Zakharov-Kuznetsov (ZK) equation
\begin{equation}\label{gZK}
u_t+u_{xxx} +u_{xyy} +u^ku_x=0.
\end{equation}
This equation is an extension of the well-known generalized Korteweg-de Vries (KdV) equation to two spatial dimensions.  The ZK equation in 3D with  $k=1$, was originally proposed by Zakharov and Kuznetsov \cite{ZK74}, to model waves in magnetized plasmas, see also \cite{KRZ}, and a rigorously justified derivation in \cite{LLS} from the Euler-Poisson system for uniformly magnetized plasma. From the local theory, it follows that solutions to the ZK equation have a maximal forward lifespan $[0, T ) $ with either $T = +\infty$ or $ T < \infty$. Also, in the case $T<+\infty$, one has $||\nabla u(t)||_{L^2(\mathbb R ^2)} \nearrow \infty $ as $t \longrightarrow T,$ although the unbounded growth of the gradient might also happen in infinite time. In \cite{Faminskii} Faminskii considered the case $k=1$ in 2D, showing  local and global well-posedness for initial data in $H^{m}(\mathbb R^2),$ $m\geq 1$, integer. His method of proof was inspired by the one given by Kenig, Ponce and Vega in \cite{Kenig} to show local well-posedness for the IVP associated with the KdV equation. To prove global results, he made use of the $L^2$ and $H^1$ conserved quantities for solutions of \eqref{gZK}. 

\medskip

In this note, we consider the gZK in two cases separately, first when $k=2,$ which corresponds to the so-called modified ZK (mZK)-equation. Notice that also in this case, mZK has a physical meaning. Indeed, it appears as an asymptotic model in the context of weakly nonlinear ion-acoustic waves in a plasma of cold ions and hot isothermal electrons with a uniform magnetic field \cite{Monro99}.  On the other hand, the case $k \geq 3$ does not seem to appear as a physically relevant model but (as the generalized KdV equation) can be used as a mathematical toy model to investigate the competition between nonlinearity and dispersion.
  	
\subsection{Setting of the model}
Consider the two-dimensional initial value problem (IVP) associated with the generalized Zakharov-Kuznetsov (gZK) equation 
\begin{equation}\label{mZK}
\begin{cases}
&u_t+u_{xxx}+u_{xyy}+u^ku_x=0, \quad (x,y)\in \mathbb{R}^2, \quad t>0, \quad k\in \mathbb{N},\\
& u(0,x,y)=u_0(x,y),
\end{cases}
\end{equation} 
where $u=u(t,x,y)$ is a real valued function. During their existence, solutions to gZK have conserved quantities. Relevant to this work is the $L^2$ norm (or mass) and the energy (or Hamiltonian):
\begin{equation*}
\begin{aligned}
M[u(t)]&=\int_{\mathbb R ^2}u^2(t)= M[u(0)],\\
E[u(t)]= \frac{1}{2}&\int_{\mathbb R^2} \left(u_x^2(t)+u_y^2(t) \right)-\frac{2}{(k-1)(k+2)}\int_{\mathbb R^2}u^{k+2}(t)=E[u(0)].
\end{aligned}
\end{equation*}	
These are well-defined quantities under appropriate conditions on the solution $u(t,x,y)$.  An important symmetry in the evolution equation \eqref{gZK} is the \textit{scaling invariance}, which states that an appropriately rescaled version of the original solution is also a solution of the equation. For the equation \eqref{mZK} it is
\begin{equation*}
u_{\lambda}(t,x,y)=\lambda^{2/k} u(\lambda^3 t, \lambda x, \lambda y), \quad \lambda>0,
\end{equation*}	
with  initial data given by $u_{\lambda}(t,x,y)=\lambda^{2/k} u_0( \lambda x, \lambda y),$ hence 
\begin{equation}\label{s_inv}
||u(0,\cdot,\cdot)||_{\dot{H}^s}=\lambda^{2/k+s-1}||u_0||_{\dot{H}^s},
\end{equation}
where $\dot{H}^s=\dot{H}^s(\mathbb{R}^2)$ denotes the homogeneous Sobolev space of order $s.$ {\color{teal} Note that when k=2, a scaling argument is useful to conclude that the blowdown rate is 1/3 in the context of the $H^1$ norm} .The scale-invariant Sobolev space for the gZK equation is $H^{s_c(k)}(\mathbb{R}^2)$, where $s_c(k) = 1-2/k$.
Therefore, we will use the  well-posedness result for \eqref{mZK} 
obtained for indices $s > s_c(k)$.
We divide the paper into two parts. The first one is built on the results of global well-posedness by Linares and Pastor in \cite{Linares2009, LinaresPastor}, for two different cases, $k=2$ and $k\geq 3$, but analyzed side by side. The second part combined the LWP results with an argument developed by  Weissler \cite{Weissler} and Colliander-Czuback-Sulem \cite{Colliander} in the context of the semilinear heat equations.

%
\medskip

Regarding to the 2D gZK equation \eqref{mZK}, Biagini-Linares \cite{BLinares}  studied the local well-posedness (LWP) for the mZK in $H^1(\mathbb R^2)$. Linares and Pastor in \cite{Linares2009} studied the IVP associated with both the ZK and mZK, in particular, they proved that the modified  ZK equation is locally well-posed for data in  $H^s(\mathbb R ^2),$ $s> 3/4$ and they also showed ill-posedness, in the sense that the data-to-solution map fails to be uniformly continuous for $s\leq 0$. So one cannot expect semilinear well-posedness in the critical space $L^2(\mathbb{R}^2)$. 
This LWP result was improved by Ribaud and Vento \cite{RV2011}, who proved LWP for data in $H^s(\mathbb R ^2),$ $s> 1/4$, for the case $k=2$. 
More recently, Kinoshita \cite{Kinoshita} established LWP at regularity $s = 1/4$, which is in fact optimal for the Picard iteration approach. Regarding global well-posedness, Linares and Pastor \cite{LinaresPastor} proved the global well-posedness in $H^s(\mathbb{R}^2)$ for $s > 53/63$, when the mass of the initial data
is smaller than the mass of the ground state in the focusing case. Bhattacharya et. al. in \cite{Bat} used the $I$-method to obtain global well-posedness in $H^s(\mathbb R^2)$ space for $s > 3/4$, thus, improving the result of Linares and Pastor \cite{LinaresPastor}. 

\medskip

For the case $k=2$, using the energy and mass conservation, together with the Gagliardo-Nirenberg inequality and its sharp constant expressed in terms of the solution mass, one has 
\[
\norm{\nabla u}_{L^2}^2 \leq \left(1-\frac{\norm{u}_{L^2}^2}{\norm{Q}_{L^2}^2} \right)^{-1}E[u]. 
\]
Thus, if $\norm{u_0}_{L^2}< \norm{Q}_{L^2},$ solutions with initial condition $u_0$ exist also globally in time,  while the {\color{black} blow-up might be possible if the initial mass $\norm{u_0}_{L^2}$ is greater or equal to that of the soliton $Q,$}  see \cite{KRS2021}. Recently, in \cite{Bat2}, it was proved mass-concentration of low-regularity blow-up solutions, hinting at an important step towards the proof of finite time blow-up. Recall the definition for blow-up solution:
\begin{defn}[Blow-up solution]
Assume $s\geq \frac14$. We say that the solution $u(t,x,y)$ to the IVP \eqref{mZK} with $u_0\in H^s(\mathbb R^2)$ blows up in finite time if there exists $0<T^{*}<\infty$ such that
\begin{equation*}
\lim_{t \uparrow T^*}\norm{u(t,x,y)}_{H^s(\mathbb R^2)}=\infty.
\end{equation*}
\end{defn}

What is the minimal rate of blow-up in gZK? What are the possible blow-up rates in gZK? This is an ongoing research question that has attracted the attention of several researchers during the past years, in particular for the case of the mZK model ($k=2$). Concerning the study of the existence of blow-up solutions, significant efforts have been made with respect to the case of the 2D mZK equation. Recall that this equation is mass critical, and therefore, can be compared to the critical generalized KdV equation in 1D. For this last equation, Martel and Merle \cite{MartelMerle} proved the existence of solutions that blow up in the $H^1(\mathbb R^2)$ norm in finite time. However, no analogous results are yet available for the  2D mZK equation. When the initial data  $u_0\in H^1(\mathbb R^2)$ Farah et. al.  \cite{Farahetal} described that there exists $\alpha > 0$ such that the solution to 2D mZK blows up in
finite or infinite time, if the energy is negative and if the mass of the initial data satisfies $\norm{Q}_{L^2(\mathbb R^2)}< \norm{u_0}_{L^2(\mathbb R^2)}< \norm{Q}_{L^2(\mathbb R^2)}+\alpha$. This is referred to as the near-threshold blow-up phenomenon for the negative energy solutions. Klein-Roudenko-Stoilov \cite{KRS2021} investigated the $H^1(\mathbb R^2)$ blow-up phenomenon for the 2D mZK equation, numerically. In particular, they conjecture that blow-up happens in finite time and that blow-up solutions have some resemblance to being self-similar, i.e., the blow-up core forms a rightward-moving self-similar type rescaled profile with the blow-up happening at finite time and spatial infinity. 

\medskip

Recently, two interesting blow-up results have been proposed for the mZK model. On the one hand, Bozgan et. al. in \cite{Bozgan} proposed a {\color{teal} analytical} proof for the existence of a finite-time blow-up solution for the mZK equation if the initial data satisfies $\norm{u_0}_{L^2} \sim \norm{Q}_{L^2(\mathbb R^2)}+\varepsilon_0$, with an additional condition on the $H^1(\mathbb R^2)$ norm of $\varepsilon_0$. Also, the authors propose that the blow-up rate is fixed and given in the interval $(\frac{1}{3},1)$. Their proof is based on Martel, Merle and Rapha\"el ideas \cite{MartelMerleRaphael} but extended to the 2D case, and requires some numerical tests for the spectral properties of related blow-up linear operators. On the other hand,  Chen et. al. in \cite{gong}, also using Martel, Merle and Rapha\"el ideas and previous spectral results from Farah et al. \cite{Farahetal}, proposed a different {\color{teal} analytical} blow up method leading to a fixed blow-up rate given in the interval $ (\frac{5}{7},\frac{5}{6}).$ This proposed blow-up rate, classified as stable, is far from the self-similar one. See {\color{teal}Figure \ref{Fig1}} for details. Also, the two mentioned recent results employ different Lyapunov functions. No minimal mass blow-up solution seems to exist in mZK \cite{gong2}, contrary to the quintic gKdV case \cite{MartelMerleRaphael2}.

\begin{figure}[h] \centering
\begin{tikzpicture}
    \draw[->] (-0.5,0) -- (10,0) node[anchor=north west] {Blow up rate mZK.};
    \draw (0,0.1) -- (0,-0.1) node[below] {0};
    \draw (1,0.1) -- (1,-0.1) node[below] {$\ge$ 7/48};
     \draw (3,0.1) -- (3,-0.1) node[below] {1/3};
    \draw (5,0.1) -- (5,-0.1) node[below] {1/2};
    \draw (7,0.1) -- (7,-0.1) node[below] {5/7};
     \draw (8,0.1) -- (8,-0.1) node[below] {$\sim$ 0.74}; 
      \draw (9,0.1) -- (9,-0.1) node[below] {5/6};
     \draw (10,0.1) -- (10,-0.1) node[below] {1};

    \filldraw (0,0) circle(1pt) node[above] {};
    \filldraw (1,0) circle(1pt) node[above] {This work};
    \filldraw (5,0) circle(1pt) node[above] {\cite{KRS2021}};
     \filldraw (8,0) circle(1pt) node[above] {\cite{gong,Bozgan}};
\end{tikzpicture}
\caption{Recent proposed blow-up rates in mZK. {\color{teal} In all cases, the initial data in space $H^1$ is considered.}}\label{Fig1}
\end{figure}

\subsection{Main Result} In this work, we provide a lower bound on the blow-up rate for solutions to mZK and gZK models in 2D. Our analysis relies on the local well-posedness results of Linares and Pastor  \cite{Linares2009,LinaresPastor} in $H^s,$ $s>3/4$. The approach is to start with some important linear estimates given by Faminskii \cite{Faminskii}, and then move on to non-linear estimates given by Linares and Pastor in \cite{Linares2009,LinaresPastor} for the mZK and gZK, respectively. In particular, in both cases, we carefully keep track of the power of time involved in the estimates, as it is central for the analysis of the lower bound for the blow-up rate. In a second stage, we will extend to 2D gZK the ideas of Colliander et. al. \cite{Colliander}, already based on previous arguments used for the heat equation by Weissler \cite{Weissler} and later extended to nonlinear Schr\"{o}dinger equations by Cazenave and Weissler \cite{CW}. More precisely, 

\begin{thm}
Consider the IVP \eqref{mZK} with initial  conditions $u_0\in H^s(\mathbb R^2)$ with $s>3/4$.  Assume that the solution $u(t,x,y)$ blows up in a finite time $T^*$ in $H^s(\mathbb R^2).$ Then,  we have the lower bound for the rate of blow-up in the corresponding Sobolev space norms, given by  
\begin{enumerate}
\item For $k=2$
\begin{equation}\label{LBR-BU}
 \norm{u(t)}_{H^s} >C(s)(T^{*}-t)^{-7/48}, \quad t\uparrow T^*.
\end{equation}

\item For $k\in \mathbb{N}$ and $k\geq 3$
\begin{equation}\label{LBR-BUk}
\| u(t)\|_{H^s} >\frac{C_k}{(T^{*}-t)^{\gamma/(k-1)}},  \quad t\uparrow T^*,
\end{equation}
where \[\frac{\gamma}{k-1}\in (0,0.041).\]
\end{enumerate}
\end{thm}

For the case of the mZK, compared with the conjectured (near) self-similar blow-up rate \cite{KRS2021}, we are somehow far from the proposed rate of decay. However, it was noticed in \cite{KRS2021} that no particular rate of decay was extremely favored from numerical experiments, leading to an unclear picture of the possible rates of decay for critical ZK. Additionally, it is expected that many possible blow-up rates are available for mZK in the case of slowly decaying slightly supercritical mass, in the same vein as previous results \cite{MartelMerleRaphael3,MartelPilod1,MartelPilod2}.

\medskip

Here is a brief description of the guidelines that we will follow in this work. First of all, the proof is based on a previous result obtained for the case of the Zakharov system by Colliander-Czuback-Sulem \cite{Colliander} and previous well-posedness results by Faminskii \cite{Faminskii} and Linares-Pastor \cite{Linares2009}. The goal is to obtain an explicit quantitative version of the local well-posedness obtained in \cite{Faminskii}, key to applying the arguments from \cite{Colliander} in the case of the mZK model. In Section \ref{linear_estimates}, we present in detail the linear estimates associated with the equation \eqref{mZK}. In the first instance, the goal is to show that we can handle the linear part in the right side of \eqref{solD}, and produce \textit{additional} powers of the existence time $T$ in the process. More precisely, we shall obtain the improved estimate \cite{Faminskii}:
\begin{lem}\label{teoF}
Let $u_0 \in H^{s}(\mathbb{R}^2)$, $s>3/4$. Then,
\begin{equation*}
\norm{U(t)u_0}_{L^2_{x}L_{yT}^{\infty}}\leq c(s)T^{1/8}\norm{u_0}_{H^s_{xy}},
\end{equation*}
with $c(s)$ is a constant depending on $s$.
\end{lem}
Next, in Section \ref{Nonlinear_estimates}  we will study nonlinear estimates for the full problem, in the spirit of Linares and Pastor \cite{Linares2009} for $k=2$ and for $k\geq 3$ we consider \cite{LinaresPastor}.  Throughout the section, we will analyze each case side by side. Our goal is to carry out all the powers in the variable $T$. Specifically, we will improve with quantitative bounds in time the following results:
\begin{thm}[Theorem 1.1 in \cite{Linares2009}]\label{LWP} Assume $k=2.$ For any $u_0\in H^s(\mathbb{R}^2),$ $s> 3/4$, there exist $T=T(\norm{u_0}_{H^s})>0$ and a unique solution of the IVP \eqref{mZK}, say $u(\cdot),$ defined in the interval $[0,T]$ such that 
\begin{align*}
& u\in C\left([0,T]; H^s(\mathbb R^2) \right),\\
& \norm{D_x^su_x}_{L^{\infty}_x L^2_{yT}} +\norm{D_y^su_x}_{L^{\infty}_x L^2_{yT}} < \infty,\\
& \norm{u}_{L^3_T L^{\infty}_{xy}} + \norm{u_x}_{L^{9/4}_T L_{xy}^{\infty}} < \infty,
\end{align*}
 and 
 \begin{equation*}
 \norm{u}_{L_x^2 L_{yT}^{\infty}}< \infty.
 \end{equation*}
\end{thm}

\begin{thm}[Theorem 1.1 in \cite{LinaresPastor}]\label{LWP2} Assume $k\in \mathbb{N}$ and  $3\leq k\leq 7.$ For any $u_0\in H^s(\mathbb{R}^2),$ $s> 3/4$, there exist $T=T(\norm{u_0}_{H^s})>0$ and a unique solution of the IVP \eqref{mZK}, say $u(\cdot),$ defined in the interval $[0,T]$ such that 
\begin{align*}
& u\in C\left([0,T]; H^s(\mathbb R^2) \right),\\
& \norm{D_x^su_x}_{L^{\infty}_x L^2_{yT}} +\norm{D_y^su_x}_{L^{\infty}_x L^2_{yT}} < \infty,\\
& \norm{u}_{L^{p_k}_T L^{\infty}_{xy}} + \norm{u_x}_{L^{12/5}_T L_{xy}^{\infty}} < \infty,
\end{align*}
 and 
 \begin{equation*}
 \norm{u}_{L_x^4 L_{yT}^{\infty}}< \infty,
 \end{equation*}
 where $p_k=\frac{12(k-1)}{7-12\gamma}$ and $\gamma \in (0,1/12).$ 
\end{thm}

 Finally, in the Section \ref{LbrBU}  we will assume that there exists a blow-up solution in finite time. Then, collecting all the previously obtained estimates, we will obtain \eqref{LBR-BU} and \eqref{LBR-BUk} for $k=2$ and $k\geq 3$, respectively.
 
 
 \subsection*{Acknowledgments} {\color{teal}Part of this work was carried out during a research visit to Georgia Tech (USA). I am deeply grateful to Professor Gong Chen for having proposed this interesting problem and for his hospitality throughout my stay. I am also indebted to Professor Claudio Muñoz for the stimulating discussions, valuable suggestions, constant support, kind comments, and advice that made it possible to obtain a high-quality version of this work.}
 
 \section{Linear estimates}\label{linear_estimates}

 \subsection*{Notation} For $\alpha \in  \mathbb{C}$, the operators $D^{\alpha}_x $ and $D^{\alpha}_y$ are defined via Fourier transform by $\widehat{D^{\alpha}_x f}(\xi \eta)=|\xi|^{\alpha}\hat f (\xi.\eta)$ and $\widehat{D^{\alpha}_y f}(\xi \eta)=|\eta|^{\alpha}\hat f (\xi ,\eta)$. The mixed space-time norm is defined by (for 1 $\leq p,q,r < \infty$)
 \begin{equation*}
 \norm{f}_{L_x^pL_y^qL_T^r}=\left(\int_{-\infty}^{+\infty} \left(\int_{-\infty}^{+\infty}\left(\int_{0}^{T}|f(t,x,y)|^rdt \right)^{q/r}dy \right)^{p/q}dx\right)^{1/p},
 \end{equation*}
with natural modification if either $p=\infty$, $q= \infty,$ or $r=\infty.$ 

\subsection{Preliminaries} Consider the Duhamel representation of the solution of the mZK equation \eqref{gZK}, that takes the form:
\begin{equation}\label{solD}
u(t,\cdot,\cdot)= U(t)u_0+\int_0^tU(t-t')(u^ku_x)(t')dt'.
\end{equation}
Define the metric spaces
\[ 
\mathcal{Y}_{T}=\{u\in C([0,T]; H^{s}(\mathbb{R}^2)) ~ | ~ \vertiii{u}< \infty \}, 
\]
and 
\[ 
\mathcal{Y}_{T}^a=\{u\in \mathcal{Y}_T; \vertiii{u}< a \},
\]
where $\mathcal{Y}_T^{a}$ will be define later.\\

Now we consider $U(t)$ above. Having these spaces and \eqref{solD} in mind, consider the simpler linear IVP
\begin{equation}\label{lineareq}
\begin{cases}
&u_t+ \partial_x \Delta u=0, \quad \quad (x,y)\in \mathbb R^2 \quad t\in \mathbb R,\\
&u(0,x,y)=u_0(x,y).
\end{cases}
\end{equation} 
The solution of the \eqref{lineareq} is given by the unitary group $\{U(t) \}_{t=-\infty}^{\infty} $:
\begin{equation}\label{solF}
u(t)=U(t)u_0(x,y)=\int_{\mathbb{R}^2} e^{i(t(\xi^3+\xi \eta^2 )+x\xi+y\eta)}\hat{u}_0(\xi,\eta)d\xi d\eta.   
\end{equation}
The estimates associated with the solution \eqref{solF} are well-known and can be seen in detail in \cite{Faminskii, Linares2009}. In particular, as we are interested in tracking the time power involved in the estimations obtained in \cite{Faminskii, Linares2009}, we will focus on the important linear estimate in Lemma \ref{teoF}, from \cite[Theorem 2.4]{Faminskii}. First, it is necessary to follow the details of the following auxiliary result, see \cite[Lemma 2.2]{Faminskii}.

\begin{lem}\label{lemma_aux}
For any $T>0$ and $k\geq 0$ there exist a constant $c>0$ and a function $H_{k,T}(x)>0$ such that
\begin{equation}\label{estimate_aux0}
\int_0^{+\infty} H_{k,T}(x)dx \leq {\color{black} c\cdot T^{1/4}} 2^{3k/2}(k+1)^2,
\end{equation}
and 
\begin{equation}\label{estimate_aux}
\left|\iint \exp(i(t\xi^3+t\xi \eta^2+x\xi+y\eta))\psi_1(\xi)\psi_2(\eta)d\xi d\eta \right| \leq H_{k,T}(|x|),
\end{equation}
for $|t|\leq T$ and $(x,y)\in \mathbb R^2,$ where $\psi_1(\xi)=\mu(a-|\xi|)$ and $\psi_2(\eta)=\mu(b-|\eta|)$ for any $a,b \leq 2^{k+1}.$
\end{lem}

The main contribution here is the explicit power $T^{1/4}$ above, which is not clear from \cite{Faminskii} and needs to be carefully worked. Here, the function $\mu(x)$ denotes a nondecreasing infinitely differentiable function defined on $\mathbb R$ such that $\mu(x)=0$ for $x\leq 0$ and $\mu(x)=1$ for $x \geq 1.$ Note that, with this definition, the function $\mu(a-|\xi|)$ is a cut-off function defined as
\[
\mu(a-|\xi|)=\begin{cases} &1, \quad |\xi|\leq a-1\\
& 0, \quad |\xi|\geq a. \end{cases} 
\]
In addition, $\mu(|\xi|-b)$ is given by 
\[\mu(|\xi|-b)=\begin{cases} &1, \quad |\xi|\geq b+1\\
& 0, \quad |\xi|\leq b. \end{cases} \]
\begin{proof}[Proof of the Lemma \ref{lemma_aux}]
We closely follow Faminskii \cite{Faminskii}. Without loss of generality, one can assume that $T \geq 1$ and $t\geq 0$.  
By $J(t,x,y)$ we denote the analog of the left-hand side in \eqref{estimate_aux} in which the function $\psi_1= \mu(a-|\xi|)$ is replaced by $\psi_1=\mu(a-|\xi|)\mu(|\xi|-1)$ (then $\psi_1(\xi)=0$ for $|\xi|\leq 1$), i.e., 
\begin{equation*}
J= \iint \exp(i(t\xi^3+t\xi \eta^2+x\xi+y\eta))\psi_1(\xi)\psi_2(\eta)d\xi d\eta.
\end{equation*}
For $|x|\leq 2^{-k/2}$ we use the  inequality 
\begin{equation}\label{caso1_a}
|J| \leq  c\cdot 2^{2k},
\end{equation}
which corresponds to the size of the domain in this region. Next, suppose that either $x\geq 2^{-k/2}$, or $x<-\max\{2^{-k/2},32t\cdot 2^{2k} \}$. Set $\varphi_1(\xi,\eta):= t\xi^3+t\xi\eta^2+x\xi;$ then $(\varphi_1)_\xi=t(3\xi^2+\eta^2)+x$ and $|(\varphi_1)_{\xi}| \geq \max \{t(3\xi^2+\eta^2),|x|/2 \}$ for
 $(\xi,\eta)\in \supp \psi_1(\xi)\psi_2(\eta)$. Integrating by parts, we obtain 
\begin{equation}\label{J:IPP}
\begin{aligned}
J&= \int \psi_2(\eta)e^{iy\eta}\left( \int e^{i\varphi_1}\psi_1(\xi)d\xi\right)d\eta\\
&=\int \psi_2(\eta)e^{iy\eta}\left(-\int\left(\frac{\psi_1(\xi)}{i(\varphi_1)'} \right)'e^{i\varphi_1}d\xi \right)d\eta.
\end{aligned}
\end{equation}
Then
\begin{equation*}\label{caso1_b}
\begin{aligned}
|J|\leq &\int \psi_{2}(\eta)\int \left|\left(\frac{\psi_1(\xi)}{i(\varphi_1(\xi))_{\xi}}\right)' \right|d\xi d\eta \\
\leq & \int \psi_{2}(\eta)\int \left( \left| \frac{\psi'_1}{\varphi'_1}\right| + \left| \frac{\psi_1 \varphi''_1}{(\varphi'_1)^2}\right| \right)d\xi d\eta \\
\lesssim  &  \int \psi_{2}(\eta)\int  (\psi'_1 |x|^{-1}+\psi_1 t |\xi||x|^{-2}) d\xi d\eta \leq  c\cdot 2^k|x|^{-1} \quad \mbox{for} \quad 2^{-k/2}\leq |x|\leq 1.
\end{aligned}
\end{equation*}
 For the remaining case, integrating by parts again in \eqref{J:IPP}, and simplifying, we get
\begin{equation}\label{caso1_c}
|J|\leq \int \psi_{2}(\xi)\int \left|\left( \frac{1}{\varphi'_1}\left(\frac{\psi_1}{\varphi'_1}\right)' \right)'\right|d\xi d\eta \leq c\cdot 2^k|x|^{-2}\quad \mbox{for}  \quad |x|\geq 1.
\end{equation}
The main difficulty comes from the remaining case $-32t \cdot 2^{2k}<x<-2^{-k/2} .$ Let us reduce $J$ as follws:
\begin{equation}\label{J1}
J= \int \Phi(y-z)\underbrace{\left(\int\left(\frac{\pi}{t|\xi|} \right)^{1/2}\exp\left(i\left( t\xi^3+x\xi-\frac{z^2}{4\xi t}+\frac{\pi}{4}\sgn \xi \right) \right)\psi_1(\xi)d\xi \right)}_{J_1}dz,
\end{equation}
where  $\Phi \equiv \mathcal{F}^{-1}_y(\psi_2).$ Let us estimate the inner integral $J_1(t,x,y)$ with respect to $\xi$ on the right-hand side in \eqref{J1}. 

\medskip

Case $z^2\geq x^2/4.$ Let us split the real axis into two parts
\begin{equation*}
\Omega_1=\{ \xi: \xi^2>|x|/(32t)\}, \quad \quad \mbox{and} \quad \quad 
\Omega_2= \{ \xi: \xi^2<|x|/(32t)\}.
\end{equation*}
From now on, the integrals $J_{1i}$ will correspond to the integral $J_1$ on each of the regions $\Omega_i$. If $\varphi(\xi)= t\xi^3+x\xi-z^2/(4\xi t)$, after the analysis for this region \cite{Faminskii} one gets
\begin{equation}\label{caso2_a}
\begin{cases}
&|J_{11}|\leq ct^{-7/12}|x|^{-1/4}\leq c_1\cdot 2^{7k/6}|x|^{-5/6},\\
& \displaystyle{ |J_{12}|\leq c \frac{(t|\xi|^{-1/2})}{|\varphi'|}\biggm|_{\partial \Omega_{2}}+c\int_{\Omega_2} \left|\left(\frac{|t\xi|^{-1/2}\psi_1}{\varphi'} \right)' \right|d\xi\leq c_12^{k/2}|x|^{-3/2}.}
\end{cases}
\end{equation}
Case $z^2\leq x^2/4.$ Let $z^2=px^2,$ $0\leq p \leq 1/4.$ Consider the following domains:
\[ 
\Omega_3=\{ \xi: \xi^2>|x|/(6t)\}, \quad   \Omega_4= \{ \xi:p|x|/(2t)< \xi^2\leq|x|/(6t)\}, \quad  \mbox{and}\quad   \Omega_5=\{ \xi: \xi^2<p|x|/(2t)\}.  
\] 
For these regions, we have 
\begin{equation}\label{caso2_b}
\begin{cases}
&|J_{13}|\leq c2^{7k/6}|x|^{-5/6},\\
&|J_{14}|\leq c_12^{k}|x|^{-3/2},\\
&|J_{15}|\leq c_1 2^{3k/2}|x|^{-1}.\\
\end{cases}
\end{equation}
Using the inequality $\int|\Phi(y)|dy\leq c(k+1)$ and the estimates \eqref{caso2_a}-\eqref{caso2_b}, from the equation \eqref{J1} we obtain
\begin{equation}\label{caso2_f}
|J|\leq c(k+1)(2^{7k/6}|x|^{-5/6}+2^{k}|x|^{-3/2}+2^{3k/2}|x|^{-1}),
\end{equation}
for $-32t\cdot 2^{2k}< x< -2^{-k/2}$ and for any $y$. Combining inequalities \eqref{caso1_a}-\eqref{caso1_c} and \eqref{caso2_f} yields the estimate \eqref{estimate_aux} for $J$. \\

\medskip
\noindent \textit{\textbf{Principal estimate:} } This is the main contribution of this work. Let us return to the structure for $J$ in \eqref{J1}, but now consider the case of 
\begin{equation}\label{J0}
J_0=\int \Phi(y-z)\left(\int\left(\frac{\pi}{t|\xi|} \right)^{1/2}\exp\left(i\left( t\xi^3+x\xi-\frac{z^2}{4\xi t}+\frac{\pi}{4}\sgn \xi \right) \right)\psi_0(\xi)d\xi\right)dz ,
\end{equation}
where $\Phi \equiv \mathcal{F}^{-1}_y(\psi_2),$ 
and  the function $\psi_0(\xi)\equiv \mu(2-|\xi|)$, instead of $\psi_1(\xi)$ ocurring with $J$ in \eqref{J1},  then we have
   \begin{equation}
|J_0|\leq c\cdot 2^{k}, \quad \quad \mbox{for} \quad |x| \leq 128T,
\end{equation}
and 
\begin{equation}
|J_0|\leq c \cdot 2^kx^{-2}, \quad \quad \mbox{for} \quad x \geq 128T \quad \mbox{or}\quad x\leq -\max(128T,32t\cdot 2^{2k}).
\end{equation}
But if $-32t\cdot 2^{2k}< x < -128T,$ then we  must estimate the inner integral $J_{0i}(t,x,z)$ with respect to $\xi$ on the right-hand side in \eqref{J0}, that contains the function $\psi_0(\xi)\equiv \mu(2-|\xi|)$ instead of $\psi_1(\xi)$ ocurring in $J$, and will be analyzed in different regions $\Omega_{0i}$, $i\in \{1,2,3,4 \} $, which will be described below.  That is
\begin{equation*}
J_{0i}:=\int_{\Omega_{0i}}\left(\frac{\pi}{t|\xi|} \right)^{1/2}\exp\left(i\left( t\xi^3+x\xi-\frac{z^2}{4\xi t}+\frac{\pi}{4}\sgn \xi \right) \right)\psi_0(\xi)d\xi. 
\end{equation*}
\textbf{Case 1:} Consider $z^2\geq \frac{x^2}{4}$. For this case, let us split the real axis into two 
\begin{equation}
\Omega_{01}= \{ \xi: \xi^2>|x|/(32t)\}, \quad \quad \Omega_{02}=\{ \xi: \xi^2<|x|/(32t) \}.
\end{equation}
Note that if $\xi\in \supp \psi_0$, then $|\xi|<|x|/32t$, i.e. $\Omega_{01}=\emptyset$
and, since $|\xi|^{-1/2}\exp(i\pi\sgn(\xi/4))=(-i\xi)^{-1/2}$, it follows that, similar to \eqref{caso2_a}, integration by parts yields 
\begin{equation*} \label{J01}
|J_0|=\left| \int \Phi(y-z)J_{02}dz \right| \leq cx^{-2}.
\end{equation*}
\textbf{Case 2:} Consider $z^2\leq \frac{x^2}{4}$. Set $\varphi(\xi)=  t\xi^3+x\xi-z^2/(4\xi t)$ and $z^2=px^2,$ $0\leq p \leq 1/4$. We divide the real axis into three parts (here $\Omega_{03}=\{\xi:\xi^2>|x|/(6t) \}=\emptyset$):
\begin{equation}
\begin{cases}
 &\Omega_{04}:= \Omega_4 \cap \{ \xi: |\xi|\geq |x|^{-1/2}\}; \quad \Omega_4= \left\{\xi: \frac{p|x|}{2t}< \xi^2 \leq \frac{|x|}{6t}\right\}, \\
& \Omega_{05}:= \Omega_5 \cap \{ \xi: |\xi|\geq |x|^{-1/2}\}; \quad \Omega_5= \left\{\xi: \xi^2< \frac{p|x|}{2t}\right\},\\
& \Omega_{06}:= \{\xi: |\xi|\leq |x|^{-1/2} \}.
 \end{cases}
\end{equation}
Let us start estimating $J_{04}$, that it is, the integral in the region $\Omega_{04}.$ Note that $\varphi'(\xi)=  3 t\xi^2 +x + \frac{z^2}{4\xi^2 t} = 3 t\xi^2 +x  + \frac{px^2}{4\xi^2 t}$ and 
\begin{equation}\label{bound_varpip}
\varphi'\left(\left(\frac{p|x|}{(2t)}\right)^{1/2}\right)
=\varphi'\left(\left(\frac{|x|}{(6t)}\right)^{1/2}\right)
=|x|\frac{3p-1}{2} \leq- \frac{|x|}{8}.
\end{equation}
Now, $\varphi'''= 6t+\frac{3z^2}{2t}\xi^{-4}>0$. It follows that $|\varphi'(\xi)|\geq \frac{|x|}{8}$ for $\xi \in \Omega_{4}$,  thus also in $\Omega_{04}$ and
\begin{equation}\label{bound_varpipp}
|\varphi''|=\left |6t\xi-\frac{1}{2}\frac{p}{t}\frac{|x|^2}{\xi^3}\right|\leq |6t\xi|+3\frac{|x|}{|\xi|}\leq c\left(t|\xi|+\left|\frac{x}{\xi}\right|\right).
\end{equation}
Also, we can write $\Omega_{04}$ as follows:
 \[ 
 \Omega_{04}= [-b,-a]\cup [a,b],\quad \mbox{with} \quad a:=  \max\left\{\left(\frac{p|x|}{2t}  \right)^{1/2},|x|^{-1/2}  \right\}\quad  \mbox{and}\quad  b:=  \left(\frac{|x|}{6t} \right)^{1/2}.
 \]
Since $ |\xi|^{-1/2} \exp(i\pi \sgn \xi/4)=(-i\xi)^{-1/2}$, then the integral to be estimated can be written as 
\begin{equation*}
J_{04}=\int_{\Omega_{04}}\frac{\pi^{1/2}}{t^{1/2}}(|\xi|)^{-1/2}\exp(i\varphi(\xi))\psi_0 d\xi.
\end{equation*}
By integration by parts 
\begin{equation*}
|J_{04}|\lesssim \pi^{1/2}\frac{|(t|\xi|)^{-1/2}\psi_0(\xi)|}{|\varphi'(\xi)|}\biggm| _{\partial \Omega_{04}} + \int_{\Omega_{04}}\left |\left(\frac{|t\xi|^{-1/2}\psi_0(\xi)}{\varphi'(\xi)} \right)'\exp(i\varphi \xi) \right|d\xi.
\end{equation*}
Let us estimate the integral of the right-hand side of the inequality. First, we have for the derivative
\begin{equation*}
\begin{aligned}
&\left|\frac{d}{d\xi}\left(\frac{(t |\xi|)^{-1/2}\psi_0(\xi)}{\varphi'(\xi)}\right)\right| \\
&=\frac{1}{\varphi'(\xi)^{2}}\left( \left( -\frac12t^{-1/2}\frac{\xi}{|\xi|}\xi^{-3/2}\psi_0+(t|\xi|)^{-1/2}\psi_0' \right)\varphi'(\xi)-(t|\xi|)^{-1/2}\psi_0\varphi''(\xi) \right).
\end{aligned}
\end{equation*}
 Using  \eqref{bound_varpip}-\eqref{bound_varpipp} we get
\begin{equation*}
\begin{aligned}
\int_{\Omega_{04}}\left |\left(\frac{|t\xi|^{-1/2}\psi_0(\xi)}{\varphi'(\xi)} \right)'\exp(i\varphi \xi) \right|d\xi \lesssim &~{} |x|^{-1}t^{-1/2}\int_{\Omega_{04}} |\xi|^{-3/2}\psi_0d\xi 
+  |x|^{-1}t^{-1/2}\int_{\Omega_{04}} |\xi|^{-1/2}\psi'_0 d\xi \\
&+  |x|^{-2}t^{-1/2}\int_{\Omega_{04}} |\xi|^{-1/2}\left(t|\xi|+\frac{|x|}{|\xi|}\right)\psi_0d\xi\\
= &~{} I_1+I_2+I_3.
\end{aligned}
\end{equation*}
We are going to estimate each integral $I_i$, $i\in \{1,2,3\}$, conveniently using the definition of the domain as follows:
\begin{equation}\label{Int1}
I_1\leq  |x|^{-1}t^{-1/2}\int_{\small{|x|^{-1/2}}}^2 |\xi|^{-3/2}d\xi\lesssim  |x|^{-1}t^{-1/2}|x|^{1/4}= t^{-1/2}|x|^{-3/4}.
\end{equation} 
Next for $I_2$ we have
\begin{equation}\label{Int2}
I_2\leq  |x|^{-1}t^{-1/2}\int_{1}^2 |\xi|^{-1/2}d\xi \lesssim |x|^{-1}t^{-1/2}|x|^{1/4}=  t^{-1/2}|x|^{-3/4}.
\end{equation}
We have used that $T\geq 1$ and $x < -128T$. For $I_3$ we get
\begin{equation}\label{Int3}
\begin{aligned}
I_3\leq&~{}  t^{1/2}|x|^{-2}\int_{|x|^{-1/2}}^{\small{\left(\frac{|x|}{6t}\right)^{1/2}}}|\xi|^{1/2}d\xi+ t^{-1/2}|x|^{-1}\int_{|x|^{-1/2}}^2|\xi|^{-3/2}d\xi\\
\lesssim &~{} t^{1/2}|x|^{-2}\left(\frac{|x|}{6t} \right)^{3/4}+   t^{-1/2}|x|^{-3/4}\leq (T^{1/4}+1) t^{-1/2}|x|^{-3/4}.
\end{aligned}
\end{equation}
Again, we have used that $T\geq 1$ and $x < -128T$. Finally, collecting estimates \eqref{Int1}-\eqref{Int3} we can conclude 
\begin{equation}\label{Ip1}
\int_{\Omega_{04}}\left |\left(\frac{|t\xi|^{-1/2}\psi_0(\xi)}{\varphi'(\xi)} \right)'\exp(i\varphi \xi) \right|d\xi \lesssim  (T^{1/4}+1)t^{-1/2}|x|^{-3/4}.
\end{equation}
Now, $\partial \Omega_{04}=\left\{-a,-b,a,b \right\}$ , with $a:= \max\left\{\left(\frac{p|x|}{2t}  \right)^{1/2},|x|^{-1/2}  \right\}$ and $b:=  \left(\frac{|x|}{6t} \right)^{1/2}$. 
Then, using \eqref{bound_varpip}, for the boundary term we have
\begin{equation}\label{borde}
\begin{aligned}
\pi^{1/2}\frac{|(t|\xi|)^{-1/2}\psi_0(\xi)|}{|\varphi'(\xi)|}\biggm| _{\partial \Omega_{04}}\lesssim &~{}  |x|^{-1} t^{-1/2} |x|^{1/4}+  |x|^{-1} t^{-1/2}\left( \frac{|x|}{6t}\right)^{-1/4} \\
\lesssim &~{}   |x|^{-3/4} t^{-1/2} +  |x|^{-1} \left( \frac{|x|}{6t}\right)^{-1/4}t^{-1/2}\lesssim T^{1/4}t^{-1/2}|x|^{-3/4}.
\end{aligned}
\end{equation}
So then, by collecting \eqref{Ip1}-\eqref{borde} we get 
\begin{equation*}\label{J04}
|J_{04}|\lesssim T^{1/4}t^{-1/2}|x|^{-3/4}.
\end{equation*}
The remaining integrals can be estimated as in \cite{Faminskii} to get:
\begin{align}
&|J_{05}|\leq c t^{-1/2}|x|^{-1/4},\nonumber \\
& |J_{06}|\leq c t^{-1/2}|x|^{-1/4},\label{J06}
\end{align}
Combining the estimates \eqref{J01}-\eqref{J06}, we obtain, similarly to \eqref{caso2_f}:
\begin{equation*}
|J_0|\leq c\cdot T^{1/4}(k+1)t^{-1/2}|x|^{-1/4}\lesssim c\cdot T^{1/4}(k+1)2^k|x|^{-3/4}. 
\end{equation*}
\end{proof}

Now we are able to give a time quantitative version of Theorem 2.4 in \cite{Faminskii}.

\begin{lem}\label{lemma_aux1} Let $u_0\in H^s$ for some $s>3/4$. Then
\begin{equation}\label{aux1}
\norm{U(t)u_0}_{L^2_xL_{yT}^{\infty}}\leq c(s)T^{1/8}\norm{u_0}_{H_{xy}^s},
\end{equation}
where $c(s)$ is a constant depending on $s$. 
\end{lem}
\begin{proof}
Let us introduce the sequence of functions $\psi_k(\xi,\eta)$ as follows: 
\begin{enumerate}
\item $\psi_0(\xi,\eta)=\mu(2-|\xi|)\mu(2-|\eta|)$, and
\item {\color{black} $\psi_k(\xi,\eta)=\mu(2^{k+1}-|\xi|)\mu(2^{k+1}-|\eta|)\mu(|\eta|-2^{k}+1)+\mu(2^{k+1}-|\xi|)\mu(|\xi|-2^{k}+1)\mu(2^k-|\eta|),$ } for $k\geq 1.$ 
\end{enumerate}
Note that, using induction, we can show that 
\[ \sum_{k=0}^{m} \psi_k(\xi,\eta) =\mu(2^{m+1}-|\xi|)\mu(2^{m+1}-|\eta|), \quad \mbox{with} \quad \supp \left( \sum_{k=0}^{m} \psi_k(\xi,\eta) \right)=[-2^{m+1}, 2^{m+1}],\]
thus  we can conclude that   $\sum_{k=0}^{\infty} \psi_k(\xi,\eta) \equiv 1.$

\medskip

For any function $f\in L^2,$ we set $B_kf=\mathcal{F}^{-1}({\color{black}\psi_k^{1/2}}\hat f).$ Note that $\norm{B_ku_0}_{L^2}\leq c\cdot 2^{-ks}\norm{u_0}_{H^s}$. If $g\in C_0^{\infty}(\mathbb R)^3$, then, by \eqref{estimate_aux}, we have
\begin{equation}\label{inq_theo}
\left| \int_{-T}^{T}(U(t-\tau))B_k^2g(\tau,\cdot,\cdot)(x,y)d\tau\right| \leq \left( H_{k,2T}(|\cdot|)*\int_{-T}^T\int|g(\tau,\cdot,y)|d\tau dy\right)(x),
\end{equation}
for $|t|\leq T.$
Let us introduce the operators  $A_k$ acting from $L^1([-T,T]; L^2) $ as follows: \[A_k=\int \chi_{T}(\tau)U(-\tau)\times B_kg(\tau,\cdot,\cdot)d\tau.\]
Then, by virtue the unitary of $U(t),$ the self-adjointness of $B_k,$ and the permutability of $U(t)$ and $B_k,$ the operator $A^{*}_k$ (acting from $L^2$ into $L^{\infty}([-T,T]; L^2) $) is specified by the formula $A^*_kv=U(t)B_kv.$

\medskip

Set $X= L_x^2(\mathbb R; L^1_{ty}([-T,T]\times \mathbb R));$ then $X^*= L_x^2(\mathbb R; L^{\infty}_{ty}([-T,T]\times \mathbb R))$ and, by virtue of \eqref{estimate_aux0} and \eqref{inq_theo}, one has
\[
\begin{aligned}
\norm{A_k^*A_kg}_{X^*}= &~{} \norm{\int_{-T}^TU(t-\tau)B^2g(\tau,\cdot,\cdot))(x,y)d\tau}_{X^*}\\
\leq &~{}  \int H_{k,2T}(|x|)dx\norm{g}_X\leq cT^{1/4}2^{3k/2}\norm{g}_X,
\end{aligned}
\]
for $g\in C_0^{\infty}(\mathbb{R}^3)$.  Therefore, by Lemma 2.1 in \cite{GV}, we have
\[ 
\norm{U(t)B_k^2v_0}_{X^*} \leq c^{1/2}T^{1/8}2^{3k/4}(k+1)\norm{B_ku_0}_{L^2}\leq c^{1/2}T^{1/8}2^{-k(s-3/4)}(k+1)\norm{u_0}_{H^s}.\]
Thus, we obtain 
\begin{equation*}
\begin{aligned}
 \norm{U(t)u_0}_{X^{*}}&\leq c\left(\sum_{k=0}^{\infty}2^{\varepsilon k}\norm{U(t)B_k^2v_0}_{X^*}^2 \right)^{1/2}\\
& \leq c\left(\sum_{k=0}^{\infty}2^{\varepsilon k}(c^{1/2}T^{1/8}2^{-k(s-3/4)}(k+1)\norm{u_0}_{H^s})^2 \right)^{1/2}\\
&\leq c\cdot T^{1/8}\left(\sum_{k=0}^{\infty}2^{\varepsilon k}2^{-k(2s-3/2)}(k+1)^2\right)^{1/2}\norm{u_0}_{H^s}\leq c_1(s)T^{1/8}\norm{u_0}_{H^s},
 \end{aligned}
\end{equation*}
if $0< \varepsilon < 2s-3/2$. 
\end{proof}

\section{Nonlinear estimates}\label{Nonlinear_estimates}
As mentioned before, the intention of this section is to trace the powers of $T$ that appear in the non-linear estimation associated to the equation \eqref{mZK}, for this, in the case $k=2$, we reanalyze the LWP result of Linares and Pastor \cite{Linares2009} and for $k\geq 3$ we taking in account the LWP result in \cite{LinaresPastor}. 
\subsection{Case $k=2:$ mZK model}
{\color{teal}For  this case consider in the definition of the space $\mathcal{Y}_{T}^a$  the norm defined as
\begin{equation}\label{norm_1}
\vertiii{u}:= \norm{u}_{L^{\infty}_TH^{s}_{xy}}+ \norm{u}_{L^{3}_T L^{\infty}_{xy}}+ \norm{u}_{L^{9/4}_T L^{\infty}_{xy}}+ \norm{D^s_xu_x}_{L^{\infty}_xL^{2}_{yT}}+  \norm{D^s_yu_x}_{L^{\infty}_xL^{2}_{yT}} + \norm{u}_{L^{2}_x L^{\infty}_{yT}},
\end{equation} }

\begin{proof}[Proof of the Theorem \ref{LWP}] We closely follow Theorem 1.1 in \cite{Linares2009}. 
Consider the integral operator:
\begin{equation*}
\Psi(u)= \Psi(u)(t)=U(t)u_0+\int_{0}^{t}U(t-t')(u^2u_x)(t')dt'.
\end{equation*}
We assume that $s>3/4$ and WLOG $T\leq 1$. We start by estimating the $H^s$-norm of $\Psi(u)$. Let $u\in \mathcal{Y}_T$. By using Minkowski's integral inequality, group properties, and H\"{o}lder inequality, we have
\begin{equation}\label{inq_1}
\begin{aligned}
\norm{\Psi(u)(t)}_{L^2_{x,y}}&\leq c \norm{u_0}_{H^s}+c\int_0^T \norm{u}_{L^2_{xy}}\norm{uu_x}_{L^{\infty}_{xy}}dt'\\
&\leq c \norm{u_0}_{H^s}+cT^{2/9}\norm{u}_{L_T^{\infty}L_{xy}^2}\norm{u}_{L_T^{3}L_{xy}^{\infty}}\norm{u_x}_{L_T^{9/4}L_{xy}^{\infty}}
\end{aligned}
\end{equation}
Recall the Leibniz rule for fractional derivatives: Let $0< \alpha < 1$ and $1<p<\infty.$ Then:
\begin{equation}\label{L_rule}
\norm{D^{\alpha}(fg)-fD^{\alpha}g-gD^{\alpha}f}_{L^p(\mathbb R)}\leq c\norm{g}_{L^{\infty}(\mathbb R)}\norm{D^{\alpha}f}_{L^p(\mathbb R)},
\end{equation}
where $D^{\alpha}$ denote  $D^{\alpha}_x$ or  $D^{\alpha}_y$.
Now, using group properties, Minkowski and H\"{o}lder inequalities, and twice the Leibniz rule  \eqref{L_rule} we have:
\begin{equation}\label{inq_2}
\begin{aligned}
\norm{D^s_x\Psi(u)(t)}_{L^2_{xy}}&\leq \norm{D_x^su_0}_{L^2_{xy}}+\int_0^T\norm{D^s_x(u^2u_x)(t')}_{L^2_{xy}}dt'\\
& \leq c\norm{u_0}_{H^s}+c\int_0^T\norm{u_x}_{L^{\infty}_{xy}}\norm{D^s_x(u^2)}_{L^2_{xy}}dt'+c\int_0^T\norm{u^2D_x^su_x}_{L^2_{xy}}dt'\\
& \leq c\norm{u_0}_{H^s}+c\int_0^T\norm{u_x}_{L^{\infty}_{xy}}\norm{u_x}_{L^{\infty}_{xy}}\norm{D^s_xu}_{L^2_{xy}}dt'+c\int_0^T\norm{u^2D_x^su_x}_{L^2_{xy}}dt'\\
& \leq c\norm{u_0}_{H^s}+c\norm{u}_{L_T^{\infty}H^s_{xy}}\int_0^T\norm{u_x}_{L^{\infty}_{xy}}\norm{u_x}_{L^{\infty}_{xy}}dt'+c\int_0^T\norm{u^2D_x^su_x}_{L^2_{xy}}dt'.\\
\end{aligned}
\end{equation}
From H\"{o}lder inequality and the argument in \eqref{inq_1}, we get 
\begin{equation}\label{inq_3}
\int_0^T\norm{u_x}_{L_{xy}^{\infty}}\norm{u_x}_{L_{xy}^{\infty}}dt'\leq T^{2/9}\norm{u}_{L_T^{3}L_{xy}^{\infty}}\norm{u_x}_{L_T^{9/4}L_{xy}^{\infty}},
\end{equation}
and 
\begin{equation}\label{inq_4}
\begin{aligned}
\int_0^T\norm{u^2D_x^su_x}_{L^2_{xy}}dt'&\leq \int_0^T\norm{u_x}_{L_{xy}^{\infty}}\norm{uD_x^su_x}_{L^2_{xy}}dt'\\
& \leq \left(\int_0^T \norm{u}_{L_{xy}^\infty}^2dt' \right)^{1/2}\norm{uD_x^su_x}_{L^2_{xyT}}\\
& \leq cT^{1/6}\norm{u}_{L_T^3L_{xy}^{\infty}}\norm{u}_{L_x^2L_{yT}^{\infty}}\norm{D_x^su_x}_{L_x^{\infty}L_{yT}^{2}}.
\end{aligned}
\end{equation}
Thus, combining \eqref{inq_2}, \eqref{inq_3} and \eqref{inq_4}, we obtain 
\begin{equation}\label{inq_5}
\begin{aligned}
\norm{D_x^s\Psi(u)(t)}_{L_{xy}^2}\leq& ~{} c\norm{u_0}_{H^s}+cT^{2/9}\norm{u}_{L_T^{\infty}H_{xy}^s}\norm{u}_{L_T^{3}L_{xy}^{\infty}}\norm{u}_{L_T^{9/4}L_{xy}^{\infty}}\\
&+ cT^{1/6}\norm{u}_{L_T^{3}L_{xy}^{\infty}}\norm{u}_{L_T^{2}L_{yT}^{\infty}}\norm{D_y^su_x}_{L_x^{\infty}L_{yT}^{2}}.
\end{aligned}
\end{equation}
Similarly for $\norm{D_y^s\Psi(u)(t)}_{L_{xy}^2},$
\begin{equation}\label{inq_6}
\begin{aligned}
\norm{D_y^s\Psi(u)(t)}_{L_{xy}^2}\leq&~{} c\norm{u_0}_{H^s}+cT^{2/9}\norm{u}_{L_T^{\infty}H_{xy}^s}\norm{u}_{L_T^{3}L_{xy}^{\infty}}\norm{u}_{L_T^{9/4}L_{xy}^{\infty}}\\
&+ cT^{1/6}||u||_{L_T^{3}L_{xy}^{\infty}}\norm{u}_{L_T^{2}L_{yT}^{\infty}}\norm{D_y^su_x}_{L_x^{\infty}L_{yT}^{2}}.
\end{aligned}
\end{equation}
Therefore, from \eqref{inq_1}, \eqref{inq_5} and \eqref{inq_6}, we deduce
\begin{equation}\label{inq_7}
\norm{\Psi(u)}_{L_T^{\infty}H^s}\leq c\norm{u_0}_{H^s}+cT^{1/6}\vertiii{u}^3.
\end{equation}
Next, from the oscillatory inequality, group properties, and the argument in \eqref{inq_1}, we get  
\begin{equation}\label{inq_8}
\begin{aligned}
\norm{\Psi(u)}_{L_T^{3}L_{xy}^{\infty}}&\leq \norm{U(t)u_0}_{L_T^3L_{xy}^{\infty}}+\norm{U(t)\left( \int_0^t U(-t')(u^2u_x)(t')dt'\right)}_{L_T^3L_{xy}^{\infty}}\\
&\leq  c\norm{u_0}_{L_{xy}^2}+c\int_0^{T}\norm{(u^2u_x)(t')}_{L^2_{xy}}dt'\\
& \leq c\norm{u_0}_{H^s}+cT^{2/9}\vertiii{u}^3.
\end{aligned}
\end{equation}
By choosing $\varepsilon \sim 1/2$ such that $1-\varepsilon/2 \leq s,$ an application of Lemma {\color{black}2.6} in \cite{Linares2009} together with arguments similar to \eqref{inq_5} yield
\begin{equation}
\begin{aligned}
\norm{\partial_x\Psi(u)}_{L_T^{4/9}L_{xy}^{\infty}}\leq&~{} \norm{U(t)\partial_xu_0}_{L_T^{4/9}L_{xy}^{\infty}}\\
& +\norm{U(t)\left(\int_0^{t}U(-t')\partial_x(u^2u_x)(t')dt' \right)}_{L_T^{4/9}L_{xy}^{\infty}}\\
\leq& ~{}c \norm{D_x^{-\varepsilon/2}\partial_xu_0}_{L_{xy}^2}+c\int_0^{T}\norm{D_x^{-\epsilon/2}\partial_x(u^2u_x)(t')}_{L^2_{xy}}dt'\\
\leq&~{} c \norm{u_0}_{H^s}+ c\int_0^T\norm{(u^2u_x)(t')}_{L^2_{xy}}dt'+c\int_0^T\norm{D_x^s(u^2u_x)(t')}_{L^2_{xy}}dt'\\
\leq & ~{}c\norm{u_0}_{H^s}+cT^{1/6}\vertiii{u}^3.
\end{aligned}
\end{equation}
Applying Lemma (2.7.i) in \cite{Linares2009}, group properties, and Minkowski and H\"{o}lder inequalitiies, we obtain 
\begin{equation*}
\begin{aligned}
\norm{D_x^s\partial_x\Psi(u)}_{L_x^{\infty}L_{yT}^2}&\leq \norm{\partial_xU(t)D_x^su_0}_{L_x^{\infty}L_{yT}^2}\\
&\leq \norm{\partial_x U(t)\left(\int_0^t U(-t')D_x^s(u^2u_x)(t')dt' \right)}_{L_x^{\infty}L_{yT}^2}\\
&\leq c\norm{D_x^su_0}_{L_{xy}^2}+c\int_0^T\norm{D_x^s(u^2u_x)(t')}_{L_{xy}^2}dt'\\
&\leq c\norm{u_0}_{H^s}+cT^{1/6}\vertiii{u}^3,
\end{aligned}
\end{equation*}
and
\begin{equation*}
\begin{aligned}
\norm{D_y^s\partial_x\Psi(u)}_{L_x^{\infty}L_{yT}^2}&\leq \norm{\partial_xU(t)D_x^su_0}_{L_x^{\infty}L_{yT}^2}\\
&\leq \norm{\partial_x U(t)\left(\int_0^t U(-t')D_x^s(u^2u_x)(t')dt' \right)}_{L_x^{\infty}L_{yT}^2}\\
&\leq c\norm{D_x^su_0}_{L_{xy}^2}+c\int_0^T\norm{D_x^s(u^2u_x)(t')}_{L_{xy}^2}dt'\\
&\leq c\norm{u_0}_{H^s}+cT^{1/6}\vertiii{u}^3.
\end{aligned}
\end{equation*} 
By an application of Lemma \ref{lemma_aux1}, Minkowski's inequality, group properties,
and arguments previously used, we get
\begin{equation}\label{inq_12}
\begin{aligned}
\norm{\Psi(u)}_{L_x^{2}L_{yT}^{\infty}}&\leq \norm{U(t)u_0}_{L_x^2L_{yT}^{\infty}}+\norm{U(t)\left( \int_{0}^t U(-t')(u^2u_x)(t')dt'\right)}_{L_x^2L_{yT}^{\infty}}\\
& \leq c(s)T^{1/8}\norm{u_0}_{H^s}+c(s)T^{1/8}\int_0^T\norm{(u^2u_x)(t')}_{H^s}dt'\\
& \leq c(s)T^{1/8}\norm{u_0}_{H^s}+c(s)T^{1/8}T^{1/6}\vertiii{u}^3.
\end{aligned}
\end{equation}
Therefore, from \eqref{inq_7}-\eqref{inq_12}, {\color{teal} we can stay that}
{\color{teal}\begin{equation}\label{estimacion_final}
\vertiii{\Psi(u)}\leq (\tilde{c}+c(s)T^{1/8})\norm{u_0}_{H^s}+(\tilde{c}(s)+c(s))T^{1/8}T^{1/6})\vertiii{u}^3.
\end{equation}}
Sine $T < 1$, setting $\tilde C$ sufficiently we can rewritte
{\color{teal}\begin{equation}\label{estimacion_final1}
\vertiii{\Psi(u)}\leq \tilde{C}(s)T^{1/8})\norm{u_0}_{H^s}+\tilde{C}(s)T^{1/8}T^{1/6})\vertiii{u}^3.
\end{equation}}

Choosing $a=2\tilde{C}(s)T^{1/8}\norm{u_0}_{H^s}$ and $T>0$ such that 
\begin{equation}
a^2\tilde{C}(s)T^{1/8}T^{1/6}\leq \frac{1}{4}.
\end{equation}
Then we get 
\begin{equation*}
\vertiii{\Psi(u)} \leq \left(\frac{1}{2}+\tilde{C}(s)^3T^{1/8+1/6}\norm{u_0}_{H^s}^2 \right)2\tilde{C}(s)T^{1/8}\norm{u_0}_{H^s}\leq \frac{3}{4}a.
\end{equation*}
Thus, we see that $\Psi: \mathcal Y_T^a \longrightarrow \mathcal Y_{T}^a$ is well defined. Moreover, similar arguments show that $\Psi$ is a contraction. 
\end{proof} 
\subsection{Case $3\leq k\leq 7,$ ($k\in\mathbb{N}$): gZK model }
{\color{teal} For this case, consider the norm defined as
\begin{equation}\label{norm_2}
\vertiii{u}:= \norm{u}_{L^{\infty}_TH^{s}_{xy}}+ \norm{u}_{L^{p_k}_T L^{\infty}_{xy}}+ \norm{u}_{L^{12/5}_T L^{\infty}_{xy}}+ \norm{D^s_xu_x}_{L^{\infty}_xL^{2}_{yT}}+  \norm{D^s_yu_x}_{L^{\infty}_xL^{2}_{yT}} + \norm{u}_{L^{4}_x L^{\infty}_{yT}},
\end{equation} 
 For each case of $k$, the constants $a, T>0$ will be chosen later.

We can see that, the norm estimates are now modulated by \[p_k=\frac{12(k-1)}{7-12\gamma}\quad  \mbox{with} \quad \gamma \in (0,1/12).\] See \cite{LinaresPastor} for full details . In this case we obtain the estimate for norm \eqref{norm_2} given by
\begin{equation}\label{estimacion_final}
\vertiii{\Psi(u)}\leq c(s)\norm{u_0}_{H^s}+c(s)T^{\gamma}\vertiii{u}^{k+1}.
\end{equation}

Choosing $a=2c(s)\norm{u_0}_{H^s}$ and $T>0$ such that 
\begin{equation}
a^kc(s)T^{\gamma}\leq \frac{1}{4}.
\end{equation}
\begin{rem}
Note that there are important differences between the norms  \eqref{norm_1} and \eqref{norm_2}.  In first place, the norm $\norm{u}_{L^{2}_x L^{\infty}_{yT}}$ that was considered for mZK was replaced by $\norm{u}_{L^{4}_x L^{\infty}_{yT}}$, for the case of the gZK, which does not consider the estimation \eqref{aux1}. Instead, use the smoothing effect of Kato type, see \cite[Lemma 2.1]{LinaresPastor}. In addition, the estimation of the norm $\norm{u}_{L^{p_k}_T L^{\infty}_{xy}}$ inserts the parameter $\gamma$ in the analysis of the well-posedness for the gZK model.
\end{rem}}
%

\section{Lower bound for the rate of blow-up}\label{LbrBU} Taking into account the estimate \eqref{estimacion_final} and  denoting 
\[
\mathcal Y(M,T)=\left\{ u: u(t=0,x)=u_0,   \vertiii{u} \leq M \right\},
\]
the local well-posedness theory is obtained by a contraction argument in the space $\mathcal Y(M,T)$, provided the smallness condition $c(s) T^{\frac18+\frac16}M^2<1/4$ is satisfied for $k=2$ and $M^kc(s)T^{\gamma}\leq \frac{1}{4}$ for $k\geq 3.$ 
We adapt arguments developed by Colliander et al. \cite{Colliander} for the Zakharov system, which were originally developed by Weissler \cite{Weissler} for the heat equation, to prove a lower bound on the rate of blow up.\\

\subsection{  mZK model ($k=2$)}
  Denote by $T^{*}$ the supremum of all $T> 0$ for which there exists a solution $u$ of the mZK satisfying
\[
 \| u \|_{L^{\infty}_TH^{s}_{xy}}+ \| u \|_{L^{3}_T L^{\infty}_{xy}}+ \| u\|_{L^{9/4}_T L^{\infty}_{xy}}+ \| D^s_xu_x\|_{L^{\infty}_xL^{2}_{yT}}+  \| D^s_yu_x\|_{L^{\infty}_xL^{2}_{yT}} + \| u \|_{L^{2}_x L^{\infty}_{yT}}< \infty.
\]
The local well-posedness theory shows $T^*>0$ and for all $t\in [0,T^*)$
  \[
\|  u(t) \|_{H^s_{xy}} < \infty.
\]
By maximality of $T^*$, it is impossible that 
\[
\|  u(t) \|_{H^s} < \infty, 
\] 
since, otherwise, it would be possible to demonstrate local well-posedness for $T_1 > T^{*}$ by taking as initial condition $u_0=u(T^*,x).$ Therefore, if $T^*< \infty,$ blow-up occurs:
\[ 
 \| u(t) \|_{H^s} \longrightarrow \infty, \quad \mbox{as} \quad t\longrightarrow T^{*}.
\]
Consider $u(t)$ posed at some time $t\in [0,T^*].$ If for some $M$
\[ 
c(s,(T-t)) \| u(t)\|_{H^s}+c(s,(T-t))(T-t)^{1/6}M^3 <M, 
\]
then, we have local well-posedness in $[t,T]$, thus $T<T^*.$ Therefore, $\forall M>0$
\[ 
\tilde{C}(s)(T^{*}-t)^{1/8} \| u(t)\|_{H^s}+\tilde{C}(s)(T^{*}-t)^{1/8}(T^*-t)^{1/6}M^3>M.
\]
Choosing $M= 2\tilde{C}(s)\| u(t)\|_{H^s}$, we have
\[
2\tilde{C}(s)M^2>(2-(T^*-t)^{1/8})(T^{*}-t)^{-(1/8+1/6)} > (T^{*}-t)^{-(1/8+1/6)}.
\]
Equivalently,
\begin{equation*}
\begin{aligned}
2\tilde{C}(s)^2\| u(t)\|_{H^s}^2 >(T^{*}-t)^{-(1/8+1/6)}, \quad \hbox{ or } \quad  \sqrt{2}\tilde{C}(s) \| u(t)\|_{H^s} >(T^{*}-t)^{-(7/48)}.
\end{aligned}
\end{equation*}
Thus, we finally get \eqref{LBR-BU}.

\medskip
{\color{teal}
\subsection{gZK model $(3\leq k\leq 7, k\in \mathbb{N})$}\label{newsect}
Applying a similar argument for gZK, again  denote by $T^{*}$ the supremum of all $T> 0$ for which there exists a solution $u$ of the gZK satisfying
\[ \norm{u}_{L^{\infty}_TH^{s}_{xy}}+ \norm{u}_{L^{p_k}_T L^{\infty}_{xy}}+ \norm{u}_{L^{12/5}_T L^{\infty}_{xy}}+ \norm{D^s_xu_x}_{L^{\infty}_xL^{2}_{yT}}+  \norm{D^s_yu_x}_{L^{\infty}_xL^{2}_{yT}} + \norm{u}_{L^{4}_x L^{\infty}_{yT}} < \infty.
\]
Now, following the same analysis,
consider $u(t)$ posed at some time $t\in [0,T^*].$ If for some $M$
\[ 
c(s) \| u(t)\|_{H^s}+c(s)(T-t)^{\gamma}M^k<M, 
\]
then, we have local well-posedness in $[t,T]$, thus $T<T^*.$ Therefore, $\forall M>0$
\[ 
c(s) \| u(t)\|_{H^s}+c(s)(T^*-t)^{\gamma}M^k>M.
\]
Choosing $M= 2c(s)\| u(t)\|_{H^s}$, we have
\[
\frac{1}{2}M+c(s)(T^*-t)^{\gamma}M^k>M
\]
Equivalently,
\begin{equation*}
\begin{aligned}
(2c(s))^k\| u(t)\|_{H^s}^{k-1} >(T^{*}-t)^{-\gamma},
\end{aligned}
\end{equation*}
or equivalently
\begin{equation}
\| u(t)\|_{H^s} >\frac{C_k}{(T^{*}-t)^{\gamma/(k-1)}},
\end{equation}
where $C_k= 1/(2c(s)^{k/(k-1)}).$ Note that , since $3\leq k \leq 7$ and $\gamma \in (0.1/12)$ we have that  must be \[\frac{\gamma}{k-1} \in (0,1/24)\] which concludes the proof of the Theorem. }


\subsection*{Funding and/or Conflicts of interests/Competing interests.}  The author declare no conflict of interest present during the elaboration of this work and its posterior publication.

{\footnotesize

}

\end{document}